\documentclass[reqno,10pt]{amsart}

\usepackage{amsmath}
\usepackage{amscd}
\usepackage{amssymb}
\usepackage{enumerate}
\usepackage{amsfonts}
\usepackage{graphicx}
\usepackage[all]{xy}
\usepackage{mathrsfs}
\usepackage{fullpage}
\usepackage{hyperref}
\usepackage{verbatim}
\usepackage{bbm}

\newcommand{\rightlim}{\mathop{\limto}}
\newcommand{\leftlim}{\mathop{\displaystyle\lim_{\longleftarrow}}}
\newcommand{\limfromn}{\leftlim\limits_{\raise3pt\hbox{$n$}}}
\newcommand{\limton}{\rightlim\limits_{\raise3pt\hbox{$n$}}}

\newcommand{\rightlimit}[1]{\mathop{\lim\limits_{\longrightarrow}}\limits%
                   _{\raise3pt\hbox{$\scriptstyle #1$}}}
\newcommand{\leftlimit}[1]{\mathop{\lim\limits_{\longleftarrow}}\limits%
                   _{\raise3pt\hbox{$\scriptstyle #1$}}}
\numberwithin{equation}{section}

\newcommand{\rar}[1]{\stackrel{#1}{\longrightarrow}}

\newcommand{\isom}{\rar{\simeq}}

\newcommand{\into}{\hookrightarrow}
\newcommand{\onto}{\twoheadrightarrow}

\newcommand{\eps}{\epsilon}

\newcommand{\cG}{{\mathcal G}}

\newcommand{\cO}{{\mathcal O}}

\newcommand{\sC}{{\mathscr C}}

\newcommand{\sG}{{\mathscr G}}

\newcommand{\fF}{{\mathfrak F}}

\newcommand{\fX}{{\mathfrak X}}

\newcommand{\fg}{{\mathfrak g}}

\newcommand{\on}{\operatorname}

\newcommand{\Ker}{\on{Ker}}
\newcommand{\Coker}{\on{Coker}}
\newcommand{\End}{\on{End}}
\newcommand{\Hom}{\on{Hom}}

\newcommand{\Spec}{\on{Spec}}

\newcommand{\id}{\on{id}}

\newcommand{\Ind}{\on{Ind}}

\renewcommand{\dot}{^{\bullet}}

\newtheorem{thm}{Theorem}[section]
\newtheorem{cor}[thm]{Corollary}
\newtheorem{lem}[thm]{Lemma}

\newtheorem{prop}[thm]{Proposition}
\newtheorem{prop-const}[thm]{Proposition-Construction}

\theoremstyle{remark}
\newtheorem{rem}[thm]{Remark}

\newcommand{\vph}{\varphi}
\newcommand{\Vect}{\on{Vect}}

\newcommand{\Sym}{\on{Sym}}
\newcommand{\Cone}{\on{Cone}}

\renewcommand{\mod}{\on{-mod}}

\newcommand{\colim}{\on{colim}}

\newcommand{\Adjoint}[4]{\xymatrix@1{#2 \ar@<.4ex>[r]^-{#1} & #3 \ar@<.4ex>[l]^-{#4}}}

\newcommand{\heart}{\heartsuit}

\newcommand{\Coh}{\on{Coh}}
\newcommand{\IndCoh}{\on{IndCoh}}
\newcommand{\QCoh}{\on{QCoh}}
\newcommand{\Sch}{\on{Sch}}

\date{\today}

\begin{document}

\begin{abstract}
We establish a criterion for sheaves on an adically complete DG scheme to be coherent.
We deduce a description of coherent sheaves on an adically complete lci singularity 
in terms of modules for a DG Lie algebra.
\end{abstract}

\author{Sam Raskin}

\title{Coherent sheaves on formal complete intersections via DG Lie algebras}
\maketitle

\setcounter{tocdepth}{1}
\tableofcontents

%%%%%%%%%%%%%

\section{Introduction}\label{s:introduction}

%%%%%%%%%%%%%

\subsection{} The aim for this note is to prove a coherence criterion for sheaves on an adically complete scheme.

We fix a ground field $k$ of characteristic zero for the remainder of this note.

\subsection{} Let us explain what the result says in the case of a complete intersection.

Let $Z$ be a complete intersection over $k$ given as a fiber product:
\[\xymatrix{
Z\ar[r]\ar[d] & X\ar[d]^f \\
\on{\Spec(k)}\ar[r]^{y} & Y
}\]
\noindent where $f$ is flat and $X$ and $Y$ are smooth schemes over $k$. Let $z\in Z$ be a closed point
and let $\widehat{Z}\in\on{Sch}$ be the adic completion of $Z$ along $z$.

Let $\Coh(\widehat{Z})$ denote the DG category of bounded complexes of coherent sheaves on $\widehat{Z}$ (its corresponding homotopy category is the triangulated
category $D^b(\Coh(\widehat{Z})^{\heart})$, the bounded derived category of finitely generated $\cO_{\widehat{Z}}$-modules). We will give an alternative
description of this DG category.

\subsection{} Recall that for $Z$ any scheme (over $k$) and $i_z:z\into Z$ a closed point, the shifted derived fiber $Li_z^*({}^RT_Z[-1])$ of the tangent complex ${}^RT_Z$
admits a canonical DG Lie algebra (aliases: $L_{\infty}$-algebra, homotopy Lie algebra) structure.

In the case of a complete intersection as above, the complex $\fg_{Z,z}$ can be computed very explicitly: it is just the two step complex in degrees $1$ and $2$:
\[
T_{X,z}\to f^*(T_{Y,y})
\]
\noindent where the differential is the derivative of $f$. While it is not so clear how to write a formula for the Lie bracket at the chain level in these terms, one can show that
at the level of cohomology the Lie bracket is the classical Hessian map:
\[
\on{Hess}(f):\Sym^2\Big(\Ker(T_{X,z}\to f^*(T_{Y,y}))\Big)\to\Coker(T_{X,z}\to f^*(T_{Y,y})).
\]
\noindent We refer to the appendix for this computation. (In fact, one can show that for a local complete intersection $Z$ as above, $\fg_{Z,z}$ is formal, so this computation
identifies $\fg_{Z,z}$ up to a non-canonical equivalence).

Let $\fg_{Z,z}\mod$ denote the DG category of modules for $\fg_{Z,z}$. Given a complex $M^{\dot}$ with the structure of $\fg_{Z,z}$-module, we can take 
its cohomology $H^*(M)$ to obtain a graded module for the graded Lie algebra $H^*(\fg_{Z,z})=H^1(\fg_{Z,z})\oplus H^2(\fg_{Z,z})$. Let
$\fg_{Z,z}\mod_{f.g.}$ denote the full DG subcategory of $\fg_{Z,z}\mod$ where we ask that $H^*(M)$ be finitely generated over
$H^*(\fg_{Z,z})$ (equivalently: as a $\Sym(H^2(\fg_{Z,z}))$-module).

\begin{thm}\label{t:qsm}
$\Coh(\widehat{Z})$ and $\fg_{Z,z}\mod_{f.g.}$ are canonically equivalent DG categories.
\end{thm}

\begin{rem}

Of course, this statement of the theorem is completely coordinate-free: its formulation makes no reference to $X$ or $Y$.

\end{rem}

\subsection{} For the remainder of this text, we will work in the setting of derived algebraic geometry as described e.g. in \cite{indcoh} and in Lurie's works.

We find it particularly convenient to use the language and notation from \cite{indcoh} and we will follow \emph{loc. cit.} in some abuses. 
In particular, ``scheme" will always mean ``derived scheme," ``category" will always mean ``$\infty$-category," (alias: $(\infty,1)$-category) ``DG category" will
mean ``stable $\infty$-category enriched over $k$," ``$\QCoh(X)$" will denote the DG category of quasi-coherent sheaves on $X$ and so on. 
We use the phrase ``classical scheme" for a scheme in the classical sense, i.e., a derived scheme which is locally the spectrum of a (usual) commutative ring.
We let $X^{cl}$ denote the classical scheme underlying a (derived) scheme $X$. 

If $\sC$ is a symmetric monoidal DG category, then by a ``Lie algebra in $\sC$" we understand an algebra in $\sC$ for the Lie operad. By
a ``Lie algebra," we understand a Lie algebra in $\on{Vect}$.

Functors between DG categories will always commute with finite colimits. By a \emph{continuous} functor, we understand a functor which
commutes with all colimits. For $\sC$ an $\infty$-category, we let $\Ind(\sC)$ denote its ind-completion, i.e., the $\infty$-category
obtained by freely adding all filtered colimits.

In addition, we impose the requirement that ``scheme" mean ``separated, quasi-compact (derived) scheme over $k$."

We will also change notation from above: all constructions will be assumed to be derived (so e.g. we will write $i^*$ where we wrote $Li^*$ and we
will write $T_X$ for the tangent complex of $X$ rather than $^RT_X$).

For $\sC$ an $\infty$-category and $X,Y\in\sC$, we let $\Hom_{\sC}(X,Y)$ denote the space (aliases: homotopy type, $\infty$-groupoid) of maps in $\sC$. 

\begin{rem}

One advantage of derived schemes for the above is very transparent: with the notation above, it allows us to omit the restriction that $f:X\to Y$ be flat at the expense
that the fiber product $Z:=X\times_Y\Spec(k)$ is a derived scheme. 

\end{rem}

\subsection{} Let us recall a bit of terminology which is somewhat non-standard.

If $\sC$ is a DG category with $t$-structure, we say an object $M\in\sC$ is \emph{connective} if $M\in\sC^{\leq 0}$ and \emph{coconnective} if $M\in\sC^{\geq 0}$ (here
and everywhere we use the cohomological indexing). We say $M$ is \emph{eventually connective} (alias: bounded above) if $M\in\sC^{-}$ and 
\emph{eventually coconnective} (alias: bounded below) if $M\in\sC^+$. For
such $\sC$, we let $\sC^{\heart}$ denote the heart of the $t$-structure. For example, $\QCoh(X)^{\heart}$ denotes the abelian category of quasi-coherent sheaves on $X$
(which we recall coincides with $\QCoh(X^{cl})^{\heart}$).

Let $Y$ be a closed subscheme of a scheme $X$. Recall from \cite{indschemes} that the formal completion $X_Y^{\wedge}$ is the prestack whose $S$-points
are the space of maps $S\to X$ such that $S\times_X(X\setminus Y)=\emptyset$. We say that $X$ is \emph{$Y$-adically complete} if the partially-defined
left adjoint to the inclusion $\Sch\into\on{PreStk}$ is defined on $X_Y^{\wedge}$ and the canonical map $X_Y^{\wedge}\to X$ realizes $X$ as the value of this functor
(and this map as the unit for the adjunction). If $Y\into X$ is a closed embedding of affine Noetherian DG schemes then this partially defined
left adjoint is defined on $X_Y^{\wedge}$ and we call its value the \emph{$Y$-adic completion of $X$} or the \emph{adic completion} of $X$ along $Y$.

We follow \cite{fishing} in saying that a scheme $X$ almost of finite type is \emph{quasi-smooth} if its cotangent complex is perfect of $\on{Tor}$-amplitude
$[-1,0]$, and more generally we say that $f:X\to Y$ is quasi-smooth if $f$ is almost of finite type and the relative tangent complex is perfect of $\on{Tor}$-amplitude $[-1,0]$. 
For $X$ classical, this coincides with the usual condition that $X$ is locally a complete intersection. 
Recall that quasi-smooth morphisms are Gorenstein (see Section \ref{ss:intro-indcoh} for the definition of Gorenstein morphism). We say that $X$ is 
\emph{smooth} if its tangent complex is a vector bundle,
i.e., perfect of $\on{Tor}$-amplitude 0; in this case, $X$ is classical and smooth in the classical sense.

\subsection{}\label{ss:intro-indcoh} Let us recall the format of ind-coherent sheaves from \cite{indcoh}.

For any Noetherian scheme $X$, let $\Coh(X)$ denote the full subcategory of $\QCoh(X)$ consisting of objects which are bounded (i.e., eventually connective and eventually coconnective) and
have cohomologies which are finitely generated in $\QCoh(X^{cl})^{\heart}$.
Let $\IndCoh(X)$ denote the DG category obtained as the ind-completion of $\Coh(X)$. 

There is a canonical functor $\Psi_X:\IndCoh(X)\to \QCoh(X)$ obtained as the unique
continuous functor inducing the canonical embedding $\Coh(X)\into \QCoh(X)$. We denote this functor by $\Psi$ when there is no risk for confusion. If $X$ is regular then 
$\Psi_X$ is an equivalence, but otherwise it is very non-conservative (i.e., it sends many objects to $0$). This failure for $\Psi$ to be an equivalence
exactly measures the difference between $\Coh(X)$ and $\on{Perf}(X)$ (the DG category of perfect complexes) and therefore may be understood as measuring in some sense the singularities of $X$.

$\IndCoh(X)$ carries a unique compactly generated $t$-structure such that $\Psi$
is $t$-exact. In fact, $\Psi$ induces an equivalence on eventually coconnective objects $\IndCoh(X)^+\isom \QCoh(X)^+$.

It is a relatively minor technical convenience to work with $\IndCoh(X)$ instead of $\Coh(X)$. For example, for a map $f:X\to Y$ almost of finite type, the functor $f^!$ 
is naturally defined as a functor $f^!:\IndCoh(Y)\to \IndCoh(X)$ and generally does not preserve coherent objects: for $X$ and $Y$ of almost finite type, this condition
amounts to being $\on{Tor}$-finite. Working systematically with the ind-completion allows us to ignore such technical conditions by forcing representability of functors and therefore allows 
for a more uniform treatment of the subject.

For $f:X\to Y$ a morphism of schemes, recall that $f_*:\QCoh(X)\to \QCoh(Y)$ is left $t$-exact and therefore maps eventually coconnective objects to eventually coconnective objects.
In particular, $f_*(\Coh(X))\subset\QCoh(Y)^+$. Therefore, the existence of the $t$-structure on $\IndCoh$ as above implies that there exists a unique continuous functor 
$f_*^{\IndCoh}:\IndCoh(X)\to\IndCoh(Y)$ such that $(\Psi_Y\circ f_*^{\IndCoh})|_{\Coh(X)}=f_*|_{\Coh(X)}$ as functors $\Coh(X)\to \QCoh(Y)$.
As in \cite{indcoh}, $f:X\to Y$ is eventually coconnective (e.g., $\on{Tor}$-finite or Gorenstein) if and only if $f_*^{\IndCoh}$ admits a left adjoint, which we denote $f^{*,\IndCoh}$ in this case.
By definition, $f$ is Gorenstein if and only if $f^{*,\IndCoh}$ (exists and) differs from $f^!$ by tensoring with a graded (i.e., cohomologically shifted) line bundle (where $\IndCoh(X)$ is regarded
as a $\QCoh(X)$-module category in the natural way). Note that \cite{fishing} Appendix E is a comprehensive reference for Gorenstein morphisms, though it contains more
precise results than are needed in this text: we only use the definition and the fact that quasi-smooth morphisms are Gorenstein. 

Finally, we recall from Section 9 of \cite{indcoh} that $\IndCoh$ is defined on any prestack locally almost of finite type.

\subsection{} We refer to \cite{fishing} Section 1 for basic facts about quasi-smooth schemes and to \emph{loc. cit.} Section 2 or \cite{dagx} Section 2 for
Lie algebras in derived algebraic geometry. For the background on ind-schemes we refer to \cite{indschemes} Sections 6-7 and for adic completions
we refer to \cite{dagxii}. Finally, we refer to \cite{indcoh} Sections 1, 3 and 4 for background on ind-coherent sheaves.

\subsection{} This note is structured as follows. Section \ref{s:main} contains the formulation and proof of the main result, Theorem \ref{t:indcoh}. Section \ref{s:quasi-smooth}
explains how to deduce the above description of coherent sheaves on the formal completion of a quasi-smooth DG scheme. The appendix
computes the graded Lie algebra associated to a complete intersection via the Hessian construction.

\subsection{Acknowledgements} It is a pleasure to thank Dennis Gaitsgory for many conversations and suggestions about this note and for his encouragement to publish it.

%%%%%%%%%%%%%

\section{Coherent sheaves}\label{s:main}

%%%%%%%%%%%%%

\subsection{} Let $X$ be an affine Noetherian scheme and let $i:Y\to X$ be a quasi-smooth closed
embedding such that $X$ is $Y$-adically complete and $Y$ is almost finite type over $k$. Let $X_Y^{\wedge}$ denote the
formal completion of $X$ along $Y$, which by Section \cite{indschemes} Section 6 is a prestack locally almost of finite type. 

We define the ``renormalized" form $\on{IndCoh}_{ren}(X_Y^{\wedge})$ of $\IndCoh(X_Y^{\wedge})$) to be the ind-completion of the full subcategory of 
$\on{IndCoh}(X_Y^{\wedge})$ of objects $\fF$ where $i^!(\fF)\in\on{Coh}(Y)$.

Our main result is the following:

\begin{thm}\label{t:indcoh}

There is a unique equivalence of categories:
\[
\on{IndCoh}(X)\isom \on{IndCoh}_{ren}(X_Y^{\wedge})
\]
\noindent compatible with the canonical functors to $\on{IndCoh}(X_Y^{\wedge})$.
\end{thm}

\begin{rem}

The requirement that $i$ be quasi-smooth is not so restrictive in derived algebraic geometry as in classical algebraic geometry.
Indeed, because $X$ is affine $Y$ can be infinitesimally thickened so that $i$ is quasi-smooth and then the above theorem
can be applied (the definition of the subcategory at which we renormalize having been modified by this procedure). In fact, in Section 
\ref{s:quasi-smooth} we will apply exactly this procedure to the embedding of the closed point of an adically complete quasi-smooth singularity
to deduce Theorem \ref{t:qsm}.

\end{rem}

\subsection{}\label{ss:condition} Consider the following condition for an object $\fF\in\on{IndCoh}(X)$ to satisfy:

\begin{description}

\item[$(*)$] $\fF$ is in the right orthogonal to the subcategory $\on{IndCoh}(X\setminus Y)\subset \on{IndCoh}(X)$
and $i^!(\fF)$ is coherent.

\end{description}

The main technical result in the proof of Theorem \ref{t:indcoh} is the following coherence criterion:

\begin{prop}\label{p:coherence}

$\fF$ is in $\Coh(X)$ if and only if $\fF$ satisfies $(*)$.

\end{prop}

The proof of Proposition \ref{p:coherence} will occupy Sections \ref{ss:coherence-start}-\ref{ss:coherence-finish}.

\begin{proof}[Proof that Proposition \ref{p:coherence} $\Rightarrow$ Theorem \ref{t:indcoh}.]

By \cite{indschemes} Section 7, the natural functor $\IndCoh(X)\to\IndCoh(X_Y^{\wedge})$ is
an equivalence when restricted to the right orthogonal to $\IndCoh(X\setminus Y)$.
Therefore, Proposition \ref{p:coherence} implies that it induces an equivalence between
$\Coh(X)$ and the compact objects in $\on{IndCoh}_{ren}(X_Y^{\wedge})$ as desired.

\end{proof}

\subsection{}\label{ss:coherence-start} We begin by showing that if $\fF\in\Coh(X)$ then it satisfies $(*)$.

Because $i$ is Gorenstein, $i^!$ differs from $i^{*,\IndCoh}$ by tensoring with a graded line bundle. Therefore,
we immediately see the  for $\fF\in\Coh(X)$, $i^!(\fF)\in\Coh(Y)$.

Now the result follows from the following lemma:

\begin{lem}\label{l:right-orth}
If $\fF\in\Coh(X)$, then $\fF$ is in the right orthogonal to $\on{IndCoh}(X\setminus Y)$.
\end{lem}

\begin{proof}

We immediately reduce to the case where $\fF\in\Coh(X)^{\heart}$.
In this case, it is a classical result that coherence for $\fF$ implies that $\fF$ is clasically $Y$-adically complete,
so in particular a countable filtered limit in $\QCoh(Y)^{\heart}$ with surjective structure maps of finitely generated modules
set-theoretically supported on $Y$.

Because $X$ is affine, the heart of the $t$-structure is closed under products. Therefore, any countable limit in $\IndCoh(X)$ (or $\QCoh(X)$) 
consisting of objects in the heart sits in cohomological degrees $[0,1]$. Because the structure maps in the limit above are surjective and because
the limit is countable, the limit in $\IndCoh(X)$ (or $\QCoh(X)$) sits in the heart and coincides with the limit as formed in the heart, i.e., $\fF$. 
This immediately gives the result.

\end{proof}

\subsection{} It remains to show that if $\fF\in\IndCoh(X)$ satisfies $(*)$, then $\fF\in\Coh(X)$. This will be done in Sections
\ref{ss:proof-start}-\ref{ss:coherence-finish}.

\subsection{}\label{ss:proof-start} We begin with the following.

\begin{lem}\label{l:coh-degs}

Suppose $\fF\in\IndCoh(X)$ is in the right orthogonal to $\IndCoh(X\setminus Y)$ and $i^!(\fF)$ is in cohomological degrees $\geq 0$.
Then $\fF$ is in cohomological degrees $\geq -1$.

\end{lem}

\begin{proof}

Suppose $\fF'\in\IndCoh(X)$ is in cohomological degrees $<-1$. It suffices to show $\Hom(\fF',\fF)=0$.
Let $j$ denote the open embedding $X\setminus Y\into X$. Since $\Hom(j_*^{\IndCoh}j^{*,\IndCoh}(\fF'),\fF)=0$, it suffices to show that
\[
\Hom\Big(\on{Cone}(\fF\to j_*^{\IndCoh}j^{*,\IndCoh}(\fF'))[-1],\fF\Big)=0.
\]
\noindent Because $\on{Cone}(\fF\to j^{\IndCoh}_*j^{*,\IndCoh}(\fF'))$ is in the left orthgonal to $\IndCoh(X\setminus Y)$ it suffices to assume $\fF'$ itself lies in this left orthgonal and sits in cohomological degrees $<0$.

By \cite{indschemes} Proposition 7.4.5, this left orthogonal is compactly generated by objects of the form $i_*^{\IndCoh}(\cG)$ for $\cG\in\Coh(Y)$.
Because the $t$-structure is compatible with filtered colimits and compact objects
are preserved under truncations, the intersection between this left orthogonal and $\IndCoh(X)^{<0}$ is compactly generated by objects $\cG\in\Coh(Y)^{<0}$. Therefore it
suffices to show:
\[
\Hom(i_*^{\IndCoh}(\cG),\fF)=0
\]
\noindent for $\cG\in\Coh(Y)^{<0}$. But this is clear since we assume $i^!(\fF)\in\IndCoh(Y)^{\geq 0}$ and $i_*^{\IndCoh}$ is $t$-exact.

\end{proof}

\begin{cor}\label{c:coh-degs}

Suppose $\fF\in\IndCoh(X)$ and $i^!(\fF)$ is eventually coconnective in $\IndCoh(Y)$. Then $\fF$ is eventually coconnective in $\IndCoh(X)$.

\end{cor}

\subsection{} Let $\fX$ be a prestack locally almost of finite type. We say $\cG\in\IndCoh(\fX)$ is \emph{eventually connective with
coherent cohomologies} if for any $S$ a DG scheme almost of finite type over $k$ and for any $\vph:S\to \fX$ the object $\vph^*(\cG)\in\QCoh(S)$ is eventually
connective with coherent cohomologies. Let $\QCoh_{coh}(\fX)^-$ denote the corresponding (non-cocomplete) DG category.

\subsection{}\label{ss:grothendieck-lurie} Let us recall the following result, which is a special case of \cite{dagxii} Theorem 5.3.2.

\begin{thm}\label{t:dcoh-}

The canonical functor:
\[
\QCoh(X)\to\QCoh(X_Y^{\wedge})
\]
\noindent induces an equivalence:
\[
\QCoh_{coh}(X)^-\to\QCoh_{coh}(X_Y^{\wedge})^-
\]
between eventually connective objects with coherent cohomologies. 

\end{thm}

Now suppose we have a sequence:
\[\xymatrix{
Y=Y_0\ar@{^(->}[r]^{\iota_0} & Y_1\ar@{^(->}[r]^{\iota_1} & Y_2\ar@{^(->}[r]^{\iota_2} & \ldots
}\]
\noindent is a sequence of closed subschemes of $X$ such that $X_Y^{\wedge}=\colim_n Y_n$ in $\on{PreStk}_{\on{Noeth}}$.
Let $i_n$ denote the closed embedding $Y_n\into X$ and let $'i_n$ denote the canonical
map $Y_n\to X_Y^{\wedge}$.

In this case, we have the (tautological) equivalence:
\[
\QCoh(X_Y^{\wedge})=\lim_n(\QCoh(Y_n))
\]
\noindent the structure maps on the are the functors $\iota_n^*$ and where the equivalence sends
$\fF\in\QCoh(X_Y^{\wedge})$ to the compatible system $\{\fF_n\}_n$ where $\fF_n:={'i_n}^*(\fF)$.

With this notation, we immediately see that the right adjoint to the canonical functor:
\[
\QCoh(X)\to\QCoh(X_Y^{\wedge})\simeq \lim_n(\QCoh(Y_n))
\]
\noindent sends a compatible system $\{\fF_n\}_n$ to the limit:
\[
\lim_n i_{n,*}(\fF_n)
\]
\noindent (where the structure maps arise via adjunction from the identification of the $*$-restriction $\fF_{n+k}$
to $Y_n$ with $\fF_n$).

In particular, Theorem \ref{t:dcoh-} gives an equivalence given by the same formula between the category $\lim_n \QCoh_{coh}(Y_n)^-$ and
$\QCoh_{coh}(X)^-$ where the notation means ``eventually connective objects with coherent cohomologies."

\subsection{}\label{ss:coh-} Now let us return to the proof of Proposition 
\ref{p:coherence}. Let us fix $\fF\in\IndCoh(X)$ satisfying $(*)$.

We claim that there exists $\{Y_n\}$ a system of almost finite type closed subschemes of $X$ as in Section \ref{ss:grothendieck-lurie} 
such that $X_Y^{\wedge}=\colim_n Y_n$ in $\on{PreStk}$ and such that $i_n^*(\Psi(\fF))$ is eventually connective with coherent cohomologies (as before, $i_n$ denotes
the embedding $Y_n\into X$).

This will be treated in Sections \ref{ss:cohminus-start}-\ref{ss:cohminus-finish}.

\subsection{}\label{ss:cohminus-start} First, assume that $i:Y\to X$ can be presented as a global complete intersection. 
That is, we assume that we have a Cartesian square:
\[\xymatrix{
Y\ar[d]\ar[r] & X\ar[d]^{f} \\
\Spec(k)\ar[r]^{0} & \mathbb{A}^s.
}\]
\noindent Then we claim that $\{Y_n\}$ can be chosen so that $Y\into Y_n$ is a composition of square-zero extensions.

Indeed, define:
\[
Y_n:=X\underset{\mathbb{A}^s}\times \Spec\big(k[X_1,\ldots,X_s]/(X_1^n,X_2^n,\ldots,X_s^n)\big).
\]
\noindent By \cite{indschemes} Section 6, we have $X_Y^{\wedge}=\colim Y_n$ as desired, and clearly $Y\into Y_n$ is a composition of square-zero extensions.

\subsection{} We continue to assume $Y\into X$ is a global complete intersection and we let $Y_n$ be as in Section \ref{ss:cohminus-start}.
Then it remains to show that $i_n^*(\Psi(\fF))$ is eventually connective with coherent cohomologies. 

Let us fix the integer $n$ for this section. Let $\cG:=i_n^*(\Psi(\fF))$ and let $\iota$ denote the embedding of $Y$ into $Y_n$.

First, note that $\iota^{*}(\cG)\in\Coh(Y)$. Indeed, this coincides with $i^*(\Psi(\fF))=\Psi(i^{\IndCoh,*}(\fF))$ which lies
in $\Coh(Y)$ because $i$ is Gorenstein (so $i^!$ differs from $i^{\IndCoh,*}$ by tensoring with a graded line bundle) and by assumption on $\fF$.

Now the claim that $\cG\in \Coh(Y_n)$ follows from the following lemma applied to $\iota:Y\into Y_n$ and $\cG$.

\begin{lem}

Let $\iota: S\into T$ be a closed embedding of Noetherian schemes which is a composition of square-zero extensions.
If $\cG\in\QCoh(T)$ satisfies $\iota^*(\cG)\in\QCoh_{coh}(S)^-$ then $\cG\in\QCoh_{coh}(T)^-$.

\end{lem}

\begin{proof}

We immediately reduce to the case where $\iota$ is a square-zero extension. In this case, we have 
the exact triangle:
\[
\iota_*(I)\to \cO_T\to \iota_*(\cO_S)\overset{+1}{\to}
\]
\noindent where $I$ is a connective $\cO_S$-module with coherent cohomologies. This gives the exact triangle:
\[
\iota_*(I)\underset{\cO_T}\otimes\cG=\iota_*(I\underset{\cO_S}{\otimes}\iota^*(\cG))\to 
\cG\to \iota_*(\cO_S)\underset{\cO_T}{\otimes}\cG=\iota_*(\iota^*(\cG))\overset{+1}{\to}
\]
\noindent immediately implying the result.

\end{proof}

\subsection{}\label{ss:cohminus-finish} Now let us complete the proof of the claim from Section \ref{ss:coh-} in the general case, i.e., when $i$ is
not assumed to be a global complete intersection.

First, recall from \cite{indschemes} Section 6 that there exists some $\{Y_n\}$ with $\colim_n Y_n=X_Y^{\wedge}$. In fact, we claim that
for any such choice of $Y_n$ we have the desired property, i.e., satisfy $i_n^*(\Psi(\fF))\in \QCoh_{coh}(Y_n)^-$.

Indeed, note that this statement is \'etale local on $X$. Therefore, we can assume that $i$ is given as a complete intersection. Then, as above, we can
choose a sequence $\{Y_m'\}$ such that for each $i_m':Y_m\into X$ we have $i_m'^*(\Psi(\fF))\in\QCoh_{coh}(Y_m)^-$ and $\colim_m Y_m'=X_Y^{\wedge}$.

But since the map $i_n:Y_n\into X$ factors through $X_Y^{\wedge}$, it must factor through some $Y_m'$. Therefore, $i_n^*(\Psi(\fF))$ is obtained
by $*$-restriction from some $i_m'^*(\Psi(\fF))$ and since this lies in $\QCoh_{coh}(Y_m)^-$ we obtain the desired result.

\subsection{}\label{ss:coherence-finish} Now we deduce Proposition \ref{p:coherence}. Let $\{Y_n\}$ be as in Section \ref{ss:coh-}.

Suppose $\fF\in\IndCoh(X)$ satisfies $(*)$. Because $\fF$ lies in the right orthogonal to $\IndCoh(X\setminus Y)$, we see
that $\Psi(\fF)$ lies in the right orthogonal to $\QCoh(X\setminus Y)$ and therefore is $Y$-adically complete, i.e.,
the morphism:
\[
\Psi(\fF)\to \lim i_{n,*}i_n^*(\Psi(\fF))
\]
\noindent is an equivalence. By Section \ref{ss:coh-} we have $i_n^*(\Psi(\fF))\in \QCoh_{coh}(Y_n)^-$. Therefore, by Section \ref{ss:grothendieck-lurie}
we see that $\Psi(\fF)\in\QCoh_{coh}(X)^-$. But by Corollary \ref{c:coh-degs} we know that $\Psi(\fF)$ is eventually coconnective in $\QCoh(X)$, giving the desired result.

%%%%%%%%%%%%%

\section{The quasi-smooth case}\label{s:quasi-smooth}

%%%%%%%%%%%%%

\subsection{} In this section, we deduce Theorem \ref{t:qsm} from Theorem \ref{t:indcoh}.

\subsection{}\label{ss:lie} Let us recall some facts about Lie algebras in derived algebraic geometry.

Let $Z$ be a DG scheme almost of finite type over $k$ and let $z\in{Z}$ be a closed point. Let 
$\fg_{Z,z}$ be the corresponding Lie algebra. Then by \cite{dagx}
there is an equivalence: $\IndCoh(Z_z^{\wedge})\simeq \fg_{Z,z}\mod$ such
that the diagram:
\[\xymatrix{
\IndCoh(Z_z^{\wedge})\ar[rr]\ar[dr]^{i_z^!} && \fg_{Z,z}\mod\ar[dl] \\
& \Vect
}\]
\noindent commutes (here $\fg_{Z,z}\mod\to\Vect$ is the canonical forgetful functor).

\begin{rem}
Since $i_{z,*}^{\IndCoh}(k(z))$ is a compact generator for $\IndCoh(Z_z^{\wedge})$, this theorem amounts to a calcuation of
$\End(i_{z,*}(k(z)))$.
\end{rem}

\subsection{} Suppose $Z$ is a quasi-smooth scheme over $k$ and let $z\in{Z}$ be a closed point. Let $\widehat{Z}$ denote the
adic completion of $Z$ at $z$. As in the introduction, let $\fg_{Z,z}\mod_{f.g.}$ denote the full subcategory of $\fg_{Z,z}\mod$
consisting of objects $M$ such that $H^*(M)$ is finitely generated over $H^*(\fg_{Z,z})$, i.e., over $\Sym(H^2(\fg_{Z,z}))$.

The main result of this section is the following:

\begin{thm}\label{t:qsm2}
There is an equivalence of categories:
\[
\IndCoh(\widehat{Z})\simeq \fg_{Z,z}\mod_{ren}:=\Ind(\fg_{Z,z}\mod_{f.g.}).
\]
\end{thm}

\begin{rem}

One can deduce similar equivalences for some related categories of sheaves. 

More precisely, let $\fg_{Z,z}\mod_{f.d.}$ denote the full subcategory of $\fg_{Z,z}\mod$
consisting of modules $M$ such that as a mere vector space $M$ is bounded cohomologically with finite dimensional cohomologies and let $\fg_{Z,z}\mod_{perf}$ denote the
category of $\fg_{Z,z}$-modules compact in $\fg_{Z,z}\mod$. Let $\fg_{Z,z}\mod_{f.d./perf}$ denote $\fg_{Z,z}\mod_{f.d.}\cap\fg_{Z,z}\mod_{perf}$.

Then we have equivalences $\QCoh(\widehat{Z})\simeq \Ind(\fg_{Z,z}\mod_{f.d.})$ and $\QCoh(Z_z^{\wedge})\simeq \Ind(\fg_{Z,z}\mod_{f.d./perf})$. 
Indeed, by Theorem \ref{t:qsm2} the former equivalence amounts to the claim that $\fF\in\Coh(\widehat{Z})$ lies in $\on{Perf}(\widehat{Z})$ if and only if $i_z^!(\fF)$ is bounded
with finite dimensional cohomologies. Let $\Xi$ denote the left adjoint to $\Psi$. Then the claim follows because we have: 
\[
i_z^!(\fF)=i_z^*(\Psi(\fF))\otimes i_z^!(\Xi(\cO_{\widehat{Z}}))
\]
\noindent and $i_z^!(\Xi(\cO_{\widehat{Z}}))$ is a shifted line because $\widehat{Z}$ is Gorenstein. This equivalence can also been seen directly from Lurie's theorem describing
$\fg_{Z,z}\mod$ (i.e., without appeal to Theorem \ref{t:indcoh}) by using the picture from \cite{indschemes} Section 7 realizing $\IndCoh(Z_z^{\wedge})$ as ind-coherent sheaves
on $Z$ set-theoretically supported at $z$, or less directly (but fundamentally by the same argument) by using the theory of singular support as developed in \cite{fishing}. 

These equivalences identify the following diagrams:
\[\xymatrix{
& \Ind(\fg_{Z,z}\mod_{f.d./perf})\ar[dl]\ar[dr] &&& \QCoh(Z_z^{\wedge})\ar[dl]\ar[dr]^{\Xi_{Z_z^{\wedge}}} \\
\Ind(\fg_{Z,z}\mod_{f.d.})\ar[dr] && \fg_{Z,z}\mod\ar[dl] & \QCoh(\widehat{Z})\ar[dr]^{\Xi_{\widehat{Z}}} && \IndCoh(Z_z^{\wedge})\ar[dl] \\
& \fg_{Z,z}\mod_{ren} &&& \IndCoh(\widehat{Z}).
}\]
\noindent Here $\QCoh(Z_z^{\wedge})\to\QCoh(\widehat{Z})$ is the left adjoint to the canonical $*$-restriction functor
and $\IndCoh(Z_z^{\wedge})\to\IndCoh(\widehat{Z})$ is the left adjoint to the canonical $!$-restriction functor.

\end{rem}

\subsection{} We will need the following result:

\begin{lem}\label{l:fg}

Suppose $A$ is an $E_{\infty}$-algebra (without any co/connectivity assumptions). Suppose that $H^*(A)$ is concentrated in even degrees (i.e., $A$ has
no non-vanishing odd cohomology groups) and
that the commutative algebra $H^*(A)=\oplus_{n\in\mathbb{Z}} H^{2n}(A)$ is a regular (in particular, Noetherian) ring. Finally, suppose that every graded-projective module
over $H^*(A)$ is graded-free.
Then an object $M\in A\mod$ is perfect if and only if $H^*(M)$ is finitely
generated over $H^*(A)$.

\end{lem}

\begin{rem}
In the setting of a (non-super) commutative $\mathbb{Z}$-graded ring $B=\oplus B_{i}$ it is not so terribly uncommon for every finitely generated graded-projective 
module to be graded-free. E.g., it holds if $B$ is Noetherian, $B_i=0$ for $i<0$ (or $i>0$) and $B_0$ is a field.
\end{rem}

\begin{proof}

First, let us show that if $M$ is perfect, then $H^*(M)$ is finitely generated over $H^*(A)$. Clearly the result holds for $M=A$ and is preserved under direct summands. Therefore,
it suffices to show that if $f:M\to M'$ is a morphism where $H^*(M)$ and $H^*(M')$ are finitely generated over $H^*(A)$, then $H^*(\Cone(f))$ is finitely generated as well.
We have a short exact sequence:
\[
0\to\Coker(H^*(f))\to H^*(\Cone(f))\to \Ker(H^{*+1}(f))\to 0
\]
\noindent and therefore the result follows immediately from the Noetherian assumption.

Now suppose that $H^*(M)$ is finitely generated over $H^*(A)$: we need to show that $M$ is perfect. We will proceed by descending induction on the homological dimension of $H^*(M)$
in the category of \emph{graded} $H^*(A)$-modules (recall that if $B$ is a graded Noetherian ring which is regular as a (non-graded) ring then its category of graded modules has 
finite homological dimension as well).
Because $H^*(M)$ is a finitely generated graded module, we can choose a map $\eps:\oplus_{i=1}^n A[r_i]\to M$ inducing a surjection $\oplus_{i=1}^n H^{*+r_i}(A)\onto H^*(M)$.
Then we claim that the homological dimension of $H^*(\Cone(\eps))$ as an $H^*(A)$-module is strictly less than the homological dimension of $H^*(M)$. Indeed,
we have the short exact sequence:
\[
0\to H^*(\Cone(\eps))\to \oplus_{i=1}^n H^{*+r_i}(A)\to H^*(M)\to 0.
\]
\noindent By induction, this reduces us to the case of homological dimension zero.

In this case, by assumption the module $H^*(M)$ is graded-free. As above, we choose a basis of $H^*(M)$ to obtain 
$\oplus_{i=1}^n A[r_i]\to M$ inducing an isomorphism at the level of cohomology, i.e., such that the map is an equivalence.

\end{proof}

\subsection{} Now let us prove Theorem \ref{t:qsm2}.

The problem is Zariski local on $Z$ and therefore we may assume $Z$ is given as a global complete intersection with $U$ and $V$ smooth:
\[\xymatrix{
Z\ar[r]\ar[d] & U\ar[d] \\
\Spec(k) \ar[r]^{i_v} & V
}\]
\noindent This defines the quasi-smooth closed embedding:
\[
\vph:\sG:=\Spec(k)\underset{V}{\times}\Spec(k)\to Z.
\]

By Theorem \ref{t:indcoh} and Section \ref{ss:lie}, it suffices to show $\fF\in\IndCoh(Z_z^{\wedge})$ 
satisfies $\vph^!(\fF)\in\Coh(\sG)$ if and only if $H^*(i_z^!(\fF))$ is finitely generated over $H^*(\fg_{Z,z})$.

By Section \ref{ss:lie}, we have the equivalences 
\[
\IndCoh(Z_z^{\wedge})\isom \fg_{Z,z}\mod
\]
\[
\IndCoh(\sG_z^{\wedge})=\IndCoh(\sG)\isom \fg_{\sG,z}\mod=\Sym(T_{V,v}[-2])\mod
\]
\noindent so that the functor $\vph^!$
corresponds to the restriction along the canonical map $T_{V,v}[-2]\to \fg_{Z,z}$. Therefore, we see that
$\vph^!(\fF)\in\Coh(\sG)$ if and only if $i_z^!(\fF)$ is perfect in $T_{V,v}[-2]\mod$. By Lemma \ref{l:fg},
this is equivalent to asking that $H^*(i_z^!(\fF))$ be finitely generated over $H^*(\Sym(T_{\sG,z}[-2]))(=\Sym(T_{V,v})$
where generators sit in degree $2$).

But since $H^1(T_{\sG,z})=T_{V,v}\to H^1(T_{Z,z})$ is a surjection, we see that this condition is equivalent to asking
that $H^*(i_z^!(\fF))$ is finitely generated over $\Sym(H^1(T_{Z,z}))$, as desired.

%%%%%%%%%%%%%

\section{Appendix: Hessian calculations}\label{s:hessian}

%%%%%%%%%%%%%

\subsection{} In this appendix, we assume a good formalism of Lie algebroids in derived algebraic geometry. Such a theory has not yet been
written down but will appear in Gaitsgory-Rozenblyum \cite{algebroids}. We will state the principal constructions we need from the theory in Section \ref{ss:prereqs}.

\subsection{} Let $f:X\to Y$ be a map between smooth schemes and let $T_{X/Y}\in\on{Perf}(X)$ be its tangent complex. This is a
Lie algebroid on $X$, so in particular $T_{X/Y}[-1]$ is equipped with a natural structure of Lie algebra in $\on{QCoh}(X)$
(this follows from the formal groupoid picture for Lie algebroids by an analysis as in \cite{fishing} Section 2).

In this appendix, we will explain how to to explicitly compute the corresponding graded Lie algebra (in the ``super" sense) $H^*(T_{X/Y}[-1])$ in $\on{QCoh}(X)^{\heart}$.

More generally, suppose $S$ is a classical scheme and $x\in X(S)$. Then the symmetric monoidal functor $x^*:\QCoh(X)\to\QCoh(S)$
gives a Lie algebra structure on $x^*(T_{X/Y}[-1])$ in $\QCoh(S)$. We will compute the corresponding
graded Lie algebra $H^*(x^*(T_{X/Y}[-1]))$ in $\on{QCoh}(S)^{\heart}$.

This graded Lie algebra has non-zero graded components only in degrees $1$ and $2$ (because $S$ was assumed classical)
and therefore its bracket is encoded entirely by a map:
\[
\Sym^2(H^0(x^*(T_{X/Y})))\to H^1(x^*(T_{X/Y})).
\]
\noindent in $\on{QCoh}(S)^{\heart}$.

\subsection{} Let $df|_x$ denote the map $x^*(T_X)\to x^*f^*(T_Y)$ obtained by applying $x^*$
to the differential of $f$.

Recall that in addition to $df|_x$, we have the classical \emph{Hessian} map (alias: second derivative):
\[
\on{Hess}_x(f):\Sym^2(H^0(x^*(T_{X/Y})))=\Sym^2(\Ker(df|_x))\to \Coker(df|_x)=H^1(x^*(T_{X/Y})).
\]
\noindent (We will review the construction of the Hessian in Section \ref{ss:hessian}).

We will show the following:

\begin{prop}\label{p:hessian}
The Lie bracket $\Sym^2(H^0(x^*(T_{X/Y})))\to H^1(x^*(T_{X/Y}))$ coincides with the Hessian.
\end{prop}

\subsection{} Before proving the proposition, we note that this computes ``explicitly" the graded Lie algebras associated to the classical points of a global complete intersection.

Indeed, let $y\in Y$ be a closed point and let $Z$ be the derived fiber product of $f$ over $y$.
Let $z$ be an $S$-point of $Z$ where $S$ is a classical scheme. We abuse notation by denoting the corresponding $S$-point of $X$ also by $z$.
Therefore, we have the diagram:
\[\xymatrix{
S\ar[dr]^{z}\ar[drr]^{z} && \\
& Z\ar[r]^{i}\ar[d]^{p_Z} & X\ar[d]^{f} \\
& \Spec(k(y))\ar[r]^{y} & Y.
}\]
\noindent Let $z$ be an $S$-point of $Z$ where $S$ is a classical scheme (this is the same as an $S$-point of $Z^{cl}$). We abuse
notation denoting the corresponding $S$-point of $X$ also by $z$.

Observe that we have the natural morphisms:
\[
H^0(z^*(T_{Z}))\rar{z^*(di)} H^0(\on{Ker}(df|_z))\subset H^0(z^*(T_X))
\]
\[
H^0(z^*p_Z^*(T_{Y,y}))=H^0(z^*f^*(T_Y))\to H^0(\on{Coker}(df|_z))\to H^1(z^*f^*(T_Z))
\]
\noindent coming from the exact triangle:
\[\xymatrix{
z^*T_Z=z^*(T_{X/Y})\ar[r] & z^*(T_X)\ar[r] & z^*(T_Y)
}\]
\noindent by taking cohomology.

Composing these maps with $i^*(\on{Hess}(f))$, we obtain a map:
\[
\Sym^2(H^0(z^*(T_Z)))\to H^1(z^*(T_Z)).
\]
\noindent By Proposition \ref{p:hessian}, this map coincides with the Lie bracket on the graded Lie algebra $H^*(z^*(T_Z[-1]))$ in
$\on{QCoh}(S)^{\heart}$.

Indeed, this follows because symmetric monoidal functor $i^*$ induces an equivalence of Lie algebras in $\on{QCoh}(Z)$:
\[
i^*(T_{X/Y}[-1])\isom T_Z[-1].
\]

\subsection{}\label{ss:hessian} Next, we recall the construction of the Hessian $\on{Hess}_x(f)$ of the map $f:X\to Y$ for $x\in X(S)$.

Let $D_X$ denote the sheaf of differential operators on $X$ and let $F_{\cdot}$ be the usual filtration of $D_X$ by order of the differential operator,
and similarly for $Y$. Recall that there is the canonical morphism $D_X\to f^*(D_Y)$ of filtered $D_X$-modules inducing $\Sym^{\dot}(df)$
on the associated graded.

We obtain a commutative diagram in $\on{QCoh}(X)^{\heart}$ where the rows are exact:
\[\xymatrix{
0\ar[r] & T_X\ar[r]\ar[d]^{df} & F_2D_X/F_0D_X\ar[r]\ar[d] & \Sym^2(T_X)\ar[r]\ar[d]^{\Sym^2(df)} & 0 \\
0\ar[r] & f^*(T_Y)\ar[r] & f^*(F_2D_Y/F_0D_Y)\ar[r] & \Sym^2 f^*(T_Y)\ar[r] & 0.
}\]
\noindent Applying $x^*$, we obtain a similar commutative diagram in $\on{QCoh}(S)^{\heart}$:
\[\xymatrix{
0\ar[r] & x^*(T_X)\ar[r]\ar[d]^{df|_x} & x^*(F_2D_X/F_0D_X)\ar[r]\ar[d] & x^*(\Sym^2(T_X))\ar[r]\ar[d]^{\Sym^2(df|_x)} & 0 \\
0\ar[r] & x^*f^*(T_Y)\ar[r] & x^*f^*(F_2D_Y/F_0D_Y)\ar[r] & \Sym^2 x^*f^*(T_Y)\ar[r] & 0.
}\]

By definition, the Hessian is obtained by composing the boundary morphism:
\[
\on{Ker}(\Sym^2(df|_x))\to \on{Coker}(df|_x)
\]
\noindent with the tautological map:
\[
\Sym^2(\on{Ker}(df|_x))\to \on{Ker}(\Sym^2(df|_x)).
\]

\begin{rem}
Suppose $X=\mathbb{A}^n$ and $Y=\mathbb{A}^m$ and let $f=(f_1,\ldots,f_m)$. Then $\on{Hess}(f)=\on{Hess}_{\id_X}(f)$ naturally factors in the obvious way through the map:
\[
\Sym^2(T_X)=\Sym^2(\cO_X[\partial_1,\ldots,\partial_n])\to f^*(T_Y)=\cO_X[\partial_1',\ldots,\partial_m']
\]
\noindent defined by the classical formula:
\[
\partial_i\cdot\partial_j\mapsto \sum_{k=1}^m \frac{\partial^2 f_j}{\partial x_j\partial x_i}\cdot\partial_j'.
\]
\end{rem}

\begin{rem}

We note that the Hessian may be also be defined via the dual description of infinitesimal neighborhoods of the diagonal in place of differential operators.

\end{rem}

\subsection{}\label{ss:prereqs} Proposition \ref{p:hessian} is based on two ``standard" compatibilities (which probably lack a reference presently).

First, suppose that $L$ is a Lie algebra in a ``nice" symmetric monoidal DG category $\mathscr{C}$ (here ``nice" means
admitting all colimits and the tensor product commutes with colimits in each variable). Let $C(L)$ be the (homological) Chevalley
complex of $L$ and let $F_{\cdot}$ denote the canonical filtration on $C(L)$. Then $\on{Gr}_i^F(C(L))$ canonically identifies
with $\Sym^i(L[1])$, so in particular we obtain a canonical map:
\[
\Lambda^2(L):=\Sym^2(L[1])[-2]\to L
\]
\noindent as from the corresponding filtration of $\Cone\big(F_0C(L)\to F_2C(L)\big)$. This map coincides (i.e., is homotopical) with
the Lie bracket for $L$.

Second, let $T$ be a scheme and let $L$ be an algebroid in $\on{Perf}(T)\subset\on{QCoh}(T)$.
Then $L[-1]$ carries a natural structure of Lie algebra in $\on{QCoh}(T)$.
With respect to this structure, the Chevalley complex $C(L[-1])$ naturally coincides as a filtered coalgebra
in $\on{QCoh}(T)$ with the enveloping algebroid $U(L)$.

\subsection{} With these preliminaries, Proposition \ref{p:hessian} follows by the following diagram chase.

Let $U(T_{X/Y})$ denote the enveloping algebroid of $T_{X/Y}$, which we regard as a mere filtered object of $\on{QCoh}(X)$.
We have a commutative diagram where all maps are compatible with filtrations and $\epsilon$ is the augmentation:
\[\xymatrix{
U(T_{X/Y})\ar[r]\ar[d]^{\epsilon} & U(T_X)\ar@{=}[r]\ar[d] & D_X\ar[d] \\
\cO_X\ar[r] & f^*(U(T_Y))\ar@{=}[r] & f^*(D_Y).
}\]
\noindent Indeed, it suffices to show that $\cO_X\to f^*(D_Y)$ is a morphism of $T_{X/Y}$-modules. To see this,
note that $\cO_Y\to D_Y$ is a map of modules for the trivial algebroid on $\cO_Y$ and
that $T_{X/Y}$ is the algebroid pull-back of the trivial algebroid. 

In particular, the natural morphism: 
\[
\Cone\big(\cO_X=F_0U(T_{X/Y})\to U(T_{X/Y})\big)\to f^*(D_Y/F_0D_Y)
\]
\noindent is canonically trivialized and therefore we obtain a natural filtered map:
\[
\Cone(F_0U(T_{X/Y})\to U(T_{X/Y})\big)\to\Cone\big(D_X/F_0D_X\to f^*(D_Y/F_0D_Y)\big)[-1].
\]

Thus, we obtain the commutative diagram between exact triangles:
\[\xymatrix{
T_{X/Y}\ar[r]\ar[d]^{\on{id}} & \Cone(F_0U(T_{X/Y})\to F_2U(T_{X/Y})\big)\ar[d]\ar[r] & \Sym^2(T_{X/Y}) \ar[d] \\
T_{X/Y}\ar[r] & \Cone\big(F_2D_X/F_0D_X\to f^*(F_2D_Y/F_0D_Y)\big)[-1] \ar[r] &
\Cone\big(\Sym^2(T_X)\rar{\Sym^2(df)} \Sym^2(f^*(T_Y))\big)[-1]
}\]
\noindent We deduce a natural commutative diagram:
\[\xymatrix{
\Sym^2(T_{X/Y})\ar[r]\ar[d] & T_{X/Y}[1]\ar[d]^{\on{id}} \\
\Cone\big(\Sym^2(T_X)\rar{\Sym^2(df)} \Sym^2(f^*(T_Y))\big)[-1] \ar[r] & T_{X/Y}[1].
}\]
\noindent By Section \ref{ss:prereqs}, the composition is equal (i.e., homotopic) to the Lie bracket.

Applying $x^*$ and passing to cohomology, we obtain the commutative diagram:
\[\xymatrix{
\Sym^2(H^0(x^*(T_{X/Y})))\ar[r]\ar[d]^{\simeq} & H^0(\Sym^2(x^*(T_{X/Y})))\ar[r]\ar[d] & H^1(x^*(T_{X/Y}))\ar[d]^{\simeq} \\
\Sym^2(\on{Ker}(df|_x))\ar[r] & \on{Ker}(\Sym^2(df|_x))\ar[r] & \on{Coker}(df|_x) \\
}\]
\noindent By definition, the composition of the bottom row computes $\on{Hess}_x(f)$.
Moreover, as above, the composition of the top row computes the Lie bracket at the level
of cohomology. This completes the computation.


\begin{thebibliography}{99}

\bibitem[AG]{fishing} D. Arinkin and D. Gaitsgory, ``Singular Support of Coherent Sheaves and the Geometric Langlands conjecture,"
arXiv:1201.6343.

\bibitem[G]{indcoh} D. Gaitsgory, ``Ind-coherent sheaves," arXiv:1105.4857.

\bibitem[GR1]{indschemes} D. Gaitsgory and N. Rozenblyum, ``DG Indschemes," arXiv:1108.1738.

\bibitem[GR2]{algebroids} D. Gaitsgory and N. Rozenblyum, ``Algebroids," forthcoming.

\bibitem[L1]{higheralgebra} J. Lurie, ``Higher algebra." Available at: http://math.harvard.edu/~lurie/papers/HigherAlgebra.pdf.

\bibitem[L2]{dagx} J. Lurie, ``Derived algebraic geometry X: Formal moduli problems." Available
at: http://math.harvard.edu/~lurie/papers/DAG-X.pdf.

\bibitem[L3]{dagxii} J. Lurie, ``Derived algebraic geometry XII: Proper morphisms, completions, and the Grothendieck existence theorem." Available at: http://math.harvard.edu/~lurie/papers/DAG-XII.pdf.

\end{thebibliography}
\end{document}